\documentclass[a4paper,11pt]{article}
\usepackage{graphicx}
\usepackage{a4}
\usepackage{amsmath, amssymb}
\usepackage{amsthm}
\newfont{\mathb}{msbm10}
\newtheorem{lemma}{Lemma}

\newtheorem{theorem}{Theorem}

\begin{document}
\pagestyle{plain}  
%
\title{Convex decompositions of point sets in the plane}
\date{}
\def\thefootnote{\fnsymbol{footnote}}
\author{Toshinori Sakai\footnotemark[1]
\and Jorge Urrutia\footnotemark[2]}
\footnotetext[1]
{Department of Mathematics and Liberal Arts Education Center, Takanawa Campus, 
Tokai University, Tokyo, Japan, 
{\tt sakai@tokai-u.jp}}
\footnotetext[2]
{Instituto de Matem\'{a}ticas,
Ciudad Universitaria,
Universidad Nacional Aut\'{o}noma de M\'{e}xico, 
M\'{e}xico D.F., M\'{e}xico, 
{\tt urrutia@matem.unam.mx}}

\maketitle

\begin{abstract}
Let $P$ be a set of $n$ points in general position on the plane. 
A set of closed convex polygons with vertices in $P$, 
and with pairwise disjoint interiors is called a convex decomposition of $P$ 
if their union is the convex hull of $P$, and 
no point of $P$ lies in the interior of the polygons. 
We show that there is a convex decomposition of $P$ 
with at most $\frac{4}{3}|I(P)|+\frac{1}{3}|B(P)|+1\le \frac{4}{3}|P|-2$ elements, 
where $B(P)\subseteq P$ is the set of points 
at the vertices of the convex hull of $P$, 
and $I(P)=P-B(P)$. 
\end{abstract}
%
%
\section{Introduction}
Let $n$ be a positive integer and 
$P$ a set of $n$ points in general position on the plane. 
We denote by ${\rm CH}(P)$ the convex hull of $P$. 

A set ${\cal \Pi}$ of closed convex polygons with vertices in $P$, 
and with pairwise disjoint interiors is called a {\em convex decomposition} of $P$ 
if their union is ${\rm CH}(P)$ and 
no point of $P$ lies in the interior of any element of $\Pi$. 
Denote by $u(P)$ the minimum number of polygons such that 
they are the elements of a convex decomposition of $P$. 
Furthermore, define $U(n)$ by the maximum of 
$u(P)$ over all sets $P$ of $n$ points in general position. 

J.\,Urrutia~\cite{1} conjectured $U(n)\le n+1$ for all $n\ge 3$. 
V.\,Neumann-Lara, E.\,Rivera-Campo and Urrutia~\cite{2} studied 
an upper-bound of $U(n)$, and 
proved $U(n)\le \frac{10n-18}{7}$ for all $n\ge 3$. 
This bound was improved by Hosono \cite{3} to 
$\left\lceil \frac{7}{5}(n-3) \right\rceil + 1$ for $n\ge 3$. 
On the other hand, O. Aichholzer and H. Krasser \cite{4} 
studied a lower-bound of $U(n)$, and proved $U(n)\ge n+2$ for $n \ge 13$. 
This bound was improved by J.\,Garc{\'\i}a-L\'opez and M.\,Nicol\'as to 
$U(n)\ge \frac{12}{11}n-2$ for $n \ge 4$. 

In this paper, we study $u(P)$ in terms of $|P|$ and 
the number of vertices of ${\rm CH}(P)$, and give 
a new upper-bound of $u(n)$. 
Let $B(P)$ denote the set of elements of $P$ 
which are the vertices of ${\rm CH}(P)$, and let 
$I(P)=P-B(P)$. 
We prove the following theorem in Section 2: 
\begin{theorem}
$\displaystyle{u(P)\le \frac{4}{3}|I(P)|+\frac{1}{3}|B(P)|+1}$.
\label{th0}
\end{theorem}

\noindent
Since $\frac{4}{3}|I(P)|+\frac{1}{3}|B(P)|+1
=\frac{4}{3}n-|B(P)|+1\le \frac{4}{3}n-2$, we also have: 
\begin{theorem}
$\displaystyle{U(n)\le \frac{4}{3}n-2}$.
\label{th1}
\end{theorem}

%
\section{Proof of Theorem~\ref{th0}}~
We prove 
\begin{equation}
u(P)\le \frac{4}{3}|I(P)|+\frac{1}{3}|B(P)|+1
\label{ineq}
\end{equation}
by induction on $|P|$. 
First consider the case where $|P|=3$. 
In this case, $|I(P)|=0, \, |B(P)|=3$ and $u(P)=1$. 
Since 
$\frac{4}{3} \! \cdot \! 0 +\frac{1}{3} \! \cdot \! 3+1=2>1$, 
(\ref{ineq}) holds. 

Next consider the case where $|P|\ge 4$. 
For this $P$, write $|I(P)|=i$, $|B(P)|=b$ and 
label the points of $B(P)$ as 
$p_0, \, p_1, \, \dots , \, p_{b-1}$ along the boundary of 
${\rm CH}(P)$ in the counter-clockwise order, 
where the indices are to be taken modulo $b$. 

Assume first that $i=1$ or $i=2$. 
In this case, 
we can verify (\ref{ineq}) without using the inductive hypothesis. 
Actually, if $i=1$, we have 
$u(P)=3 \le \frac{4}{3} \! \cdot \! 1 +\frac{1}{3}b+1$ (Figure~\ref{cd01}(a)); 
and 
if $i=2$, $u(P)=4\le \frac{4}{3} \! \cdot \! 2 +\frac{1}{3}b+1$ 
(Figure~\ref{cd01}(b)). 
\begin{center}
 \begin{figure}[htbp]
 \begin{minipage}[b]{.48\linewidth}
 \, \, \, 
  \begin{center}\includegraphics[width=.8\linewidth]{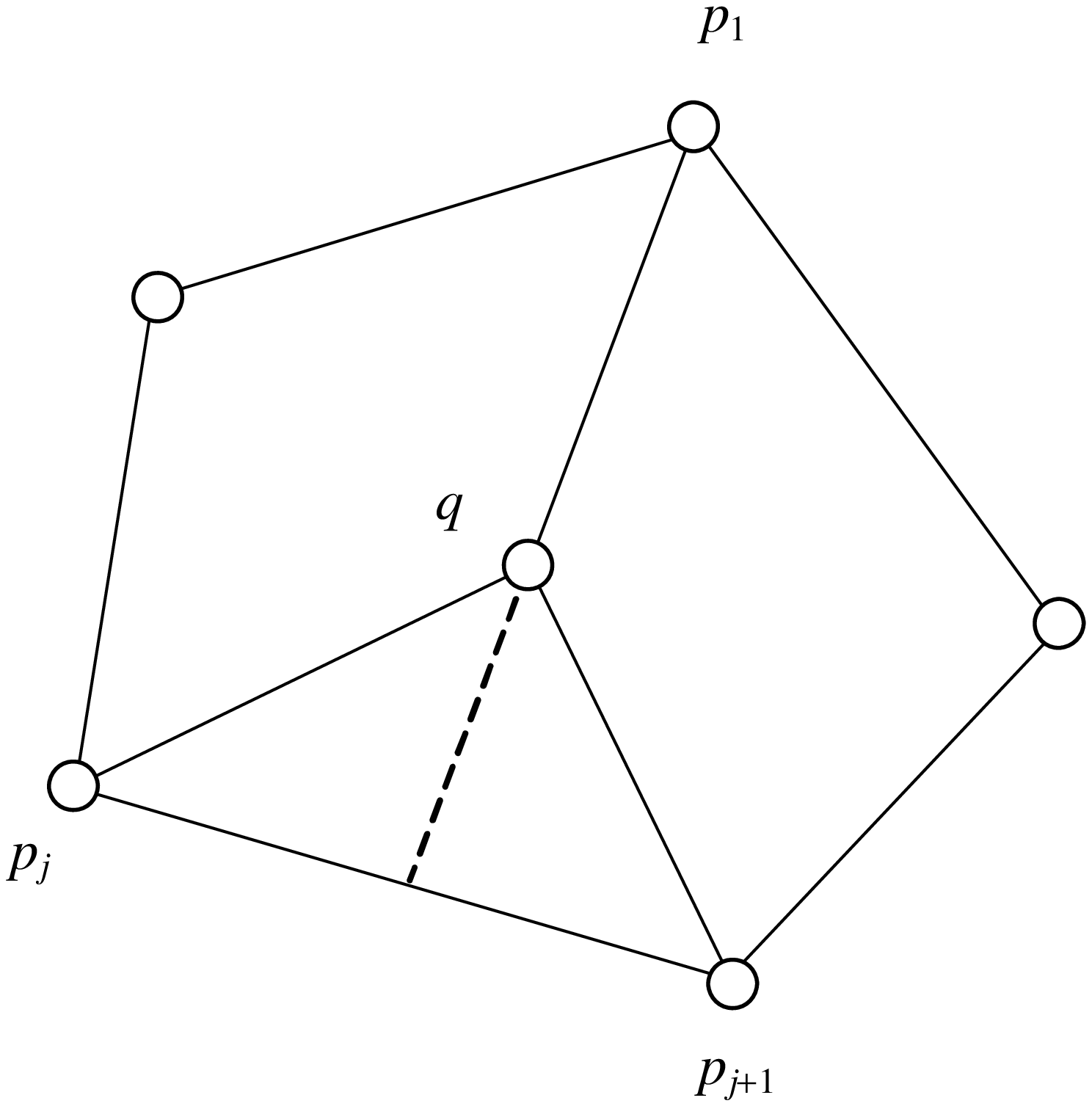}\\
  (a)
  \end{center}
  \end{minipage}
   \begin{minipage}[b]{.48\linewidth}
 \, \, \, 
  \begin{center}\includegraphics[width=.8\linewidth]{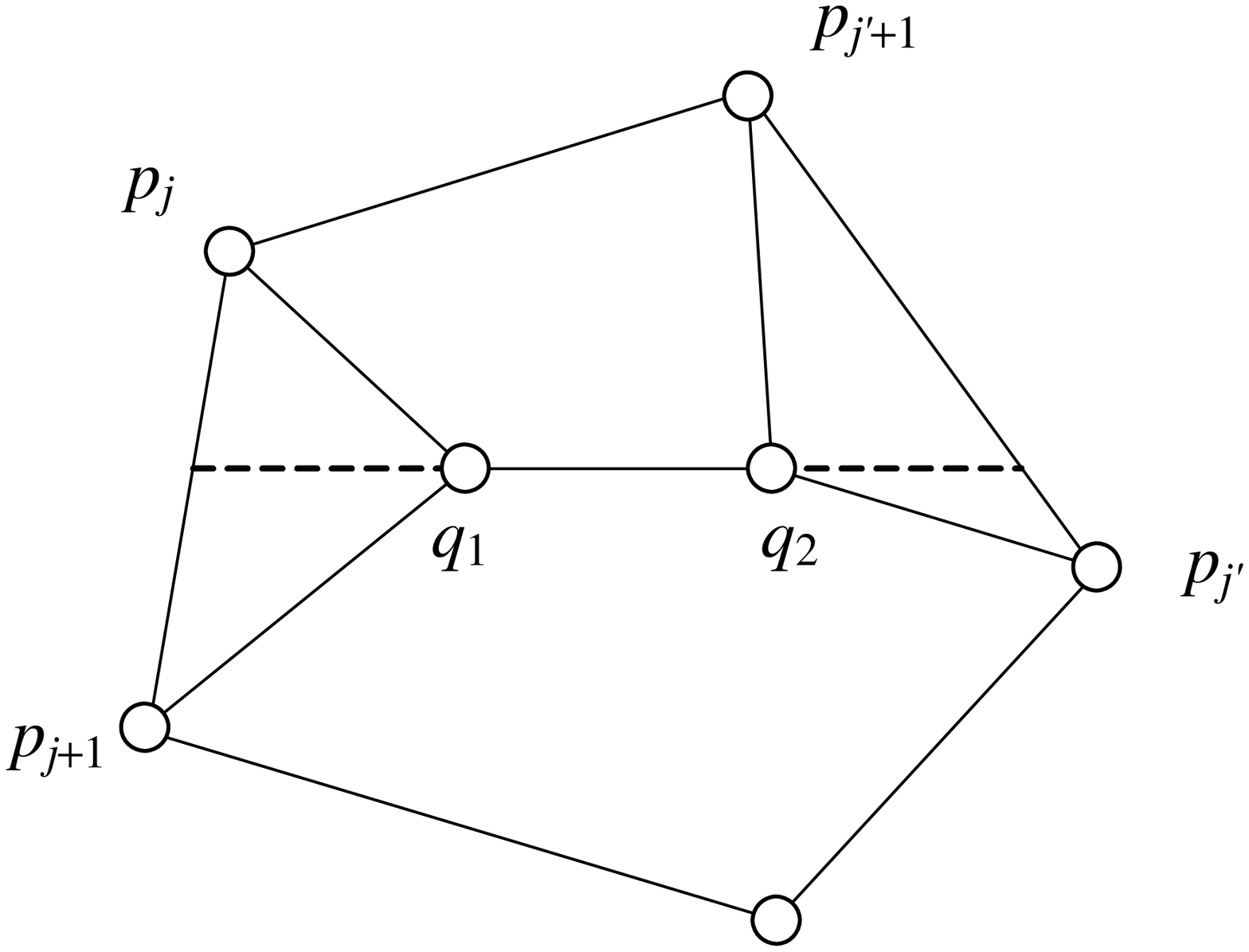}\\
  (b)
  \end{center}
  \end{minipage}
  \caption{The cases where (a) $i=1$ and (b) $i=2$.}
  \label{cd01}
  \end{figure}
\end{center}

Thus assume that
\begin{equation}
i\ge 3.
\label{0}
\end{equation}

For three points $p, q, r$ in general position, 
we denote by $\triangle pqr$ the triangle with vertices $p, q, r$, 
and for a subset $K\subset {\mathbb R}^2$, we denote by ${\rm Int}(K)$
the interior of $K$. 

\medskip
\noindent
{\bf Case 1.}~~${\rm Int}(\triangle p_{j-1}p_jp_{j+1}) \cap P=\emptyset $ 
for some $j$ (Figure~\ref{cd02}):

Take $j$ with ${\rm Int}(\triangle p_{j-1}p_jp_{j+1}) \cap P=\emptyset $. 
By inductive hypothesis, there is a convex decomposition $\Pi _0$ of 
$P-\{p_j\}$ with $u(P-\{p_j\})\le 
\frac{4}{3}i+\frac{1}{3}(b-1)+1<\frac{4}{3}i+\frac{1}{3}b+1$ elements. 
Let $\mathcal P$ be an element of $\Pi _0$ which has the segment 
$p_{j-1}p_{j+1}$ on its boundary. 
Then $\triangle p_{j-1}p_jp_{j+1}$ can be combined with $\mathcal P$ 
to form a single convex polygon. 
Thus 
$u(P)\le u(P-\{p_j\})<\frac{4}{3}i+\frac{1}{3}b+1$, as desired. 
%
\begin{center}
 \begin{figure}[htbp]
 \begin{minipage}[b]{.48\linewidth}
 \, \, \, 
  \begin{center}\includegraphics[width=.8\linewidth]{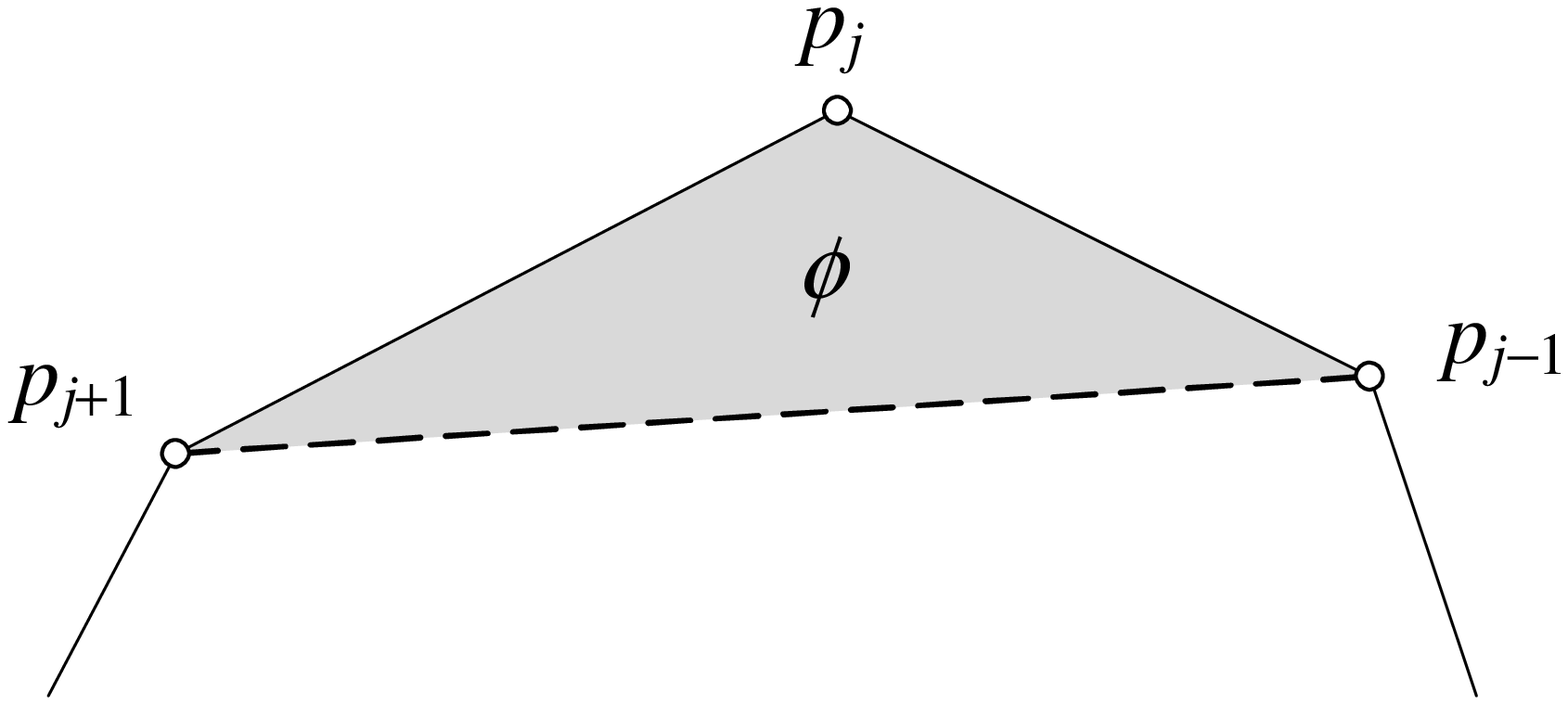}
  \caption{${\rm Int}(\triangle p_{j-1}p_jp_{j+1}) \cap P=\emptyset $.}
  \label{cd02}
  \end{center}
  \end{minipage}
   \begin{minipage}[b]{.48\linewidth}
 \, \, \, 
  \begin{center}\includegraphics[width=.8\linewidth]{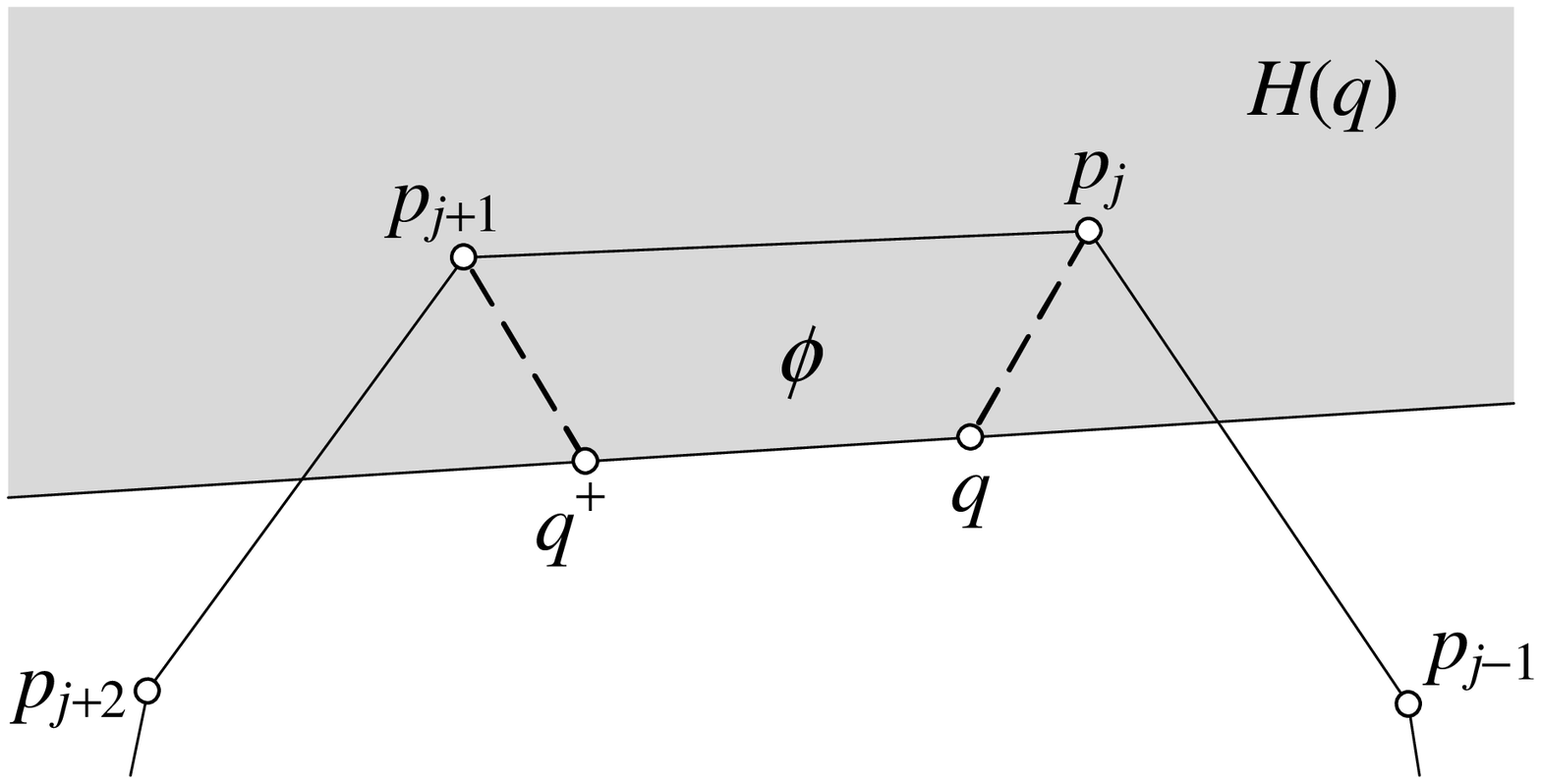}
  \caption{\bf $|H(q)\cap B(P)|=2$.}
  \label{cd03}
  \end{center}
  \end{minipage}
 \end{figure}
\end{center}

\medskip 
Now we may assume that 
\begin{equation}
\mbox{${\rm Int}(\triangle p_{j-1}p_jp_{j+1}) \cap P\ne \emptyset $ 
\quad for all $j$.}
\label{1}
\end{equation}

For $q\in B(I(P))$, let $q^+$ denote the element of $B(I(P))$ 
which occurs immediately after $q$ 
on the boundary of ${\rm CH}(I(P))$ 
in the counter-clockwise order, 
and denote by $H(q)$ the open half-plane 
which has $q$ and $q^+$ on its boundary and 
no point of $I(P)$ in its interior. 
It follows from (\ref{1}) that 
$|H(q)\cap B(P)|\le 2$ for all $q\in B(I(P))$.

\medskip
\noindent
{\bf Case 2.}~~$|H(q)\cap B(P)|=2$ for 
some $q\in B(I(P))$ (Figure~\ref{cd03}):

Take $q\in B(I(P))$ with $|H(q)\cap B(P)|=2$. 
For this $q$, let $p_j$ and $p_{j+1}$ be the two elements of $H(q) \cap B(P)$, 
and let $P'=P-\{p_j, p_{j+1}\}, \, m=|B(P')\cap I(P)|$ and 
$q_0=p_{j-1}$. 
Label the elements of $B(P')\cap I(P)$ as $q_1, q_2, \dots , q_m$ 
in such a way that $q_0, q_1, q_2, \ldots , q_m$ occur on 
the boundary of ${\rm CH}(P')$ in the counter-clockwise order, 
and let $q_{m+1}=p_{j+2}$. 
Furthermore, we let $l$ be the index such that $q_l=q$ (Figure~\ref{cd04}) and 
${\cal P}$ the quadrilateral $q_lp_jp_{j+1}q_{l+1}$, which is convex. 
Set 
$$
\Pi_0=\{{\cal P}\} \cup 
\bigl( \cup _{k=0}^{l-1} \{ \triangle q_kp_jq_{k+1} \}\bigr)
\cup \bigl( \cup _{k=l+1}^{m} \{ \triangle q_kp_{j+1}q_{k+1} \}\bigr).
$$
We have $|\Pi _0|=m+1$. 
%
\begin{center}
 \begin{figure}[htbp]
 \begin{minipage}[b]{.48\linewidth}
 \, \, \, 
  \begin{center}\includegraphics[scale=.3]{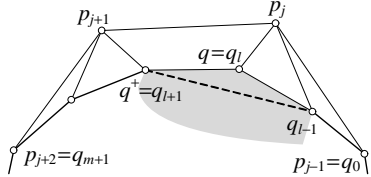}\\
  (a)~${\rm Int}(\triangle p_{j-1}p_jp_{j+1})\cap P = \emptyset$.
  \end{center}
  \end{minipage}
   \begin{minipage}[b]{.48\linewidth}
 \, \, \, 
  \begin{center}\includegraphics[scale=.3]{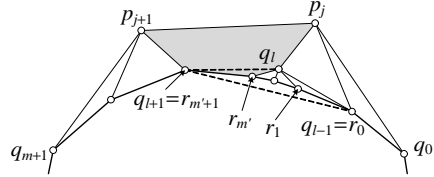}\\
  (b)~${\rm Int}(\triangle p_{j-1}p_jp_{j+1})\cap P \ne \emptyset$.
  \end{center}
  \end{minipage}
  \caption{Two subcases of Case 2.}
  \label{cd04}
  \end{figure}
\end{center}

First assume that 
${\rm Int}(\triangle q_{l-1}q_lq_{l+1}) \cap P=\emptyset$ (Figure~\ref{cd04}(a)). 
Let $\Pi _1$ be a convex decomposition of $P'$ with $u(P')$ elements. 
Then by the similar argument as Case 1, we have 
\begin{eqnarray*}
u(P') & \le & u(P'-\{q_l\})\\
 & \le & \frac{4}{3}(i-m)+\frac{1}{3}[(b-2+(m-1)]+1
 \, \, = \, \, \frac{4}{3}i+\frac{1}{3}b-m.
\end{eqnarray*}
Hence $\Pi _1 \cup \Pi _0$ is a convex decomposition of $P$ 
with $u(P')+(m+1)\le \frac{4}{3}i+\frac{1}{3}b+1$ elements. 

Next assume that 
${\rm Int}(\triangle q_{l-1}q_lq_{l+1}) \cap P \ne \emptyset$ (Figure~\ref{cd04}(b)). 
Let $P''=P'-\{q_l\}, \, m'=|B(P'')\cap I(P')|$ and $r_0=q_{l-1}$. 
Label the elements of $B(P'')\cap I(P')$ as $r_1, r_2, \dots , r_{m'}$ 
in such a way that $r_0, r_1, r_2, \ldots , r_{m'}$ occur on 
the boundary of ${\rm CH}(P'')$ in the counter-clockwise order, 
and let $r_{m'+1}=q_{l+1}$. 
Denote by $\Pi _1$ a convex decomposition of $P''$ with $u(P'')$ elements. 
We have 
\begin{eqnarray*}
u(P'') & \le & \frac{4}{3}(i-m-m')+\frac{1}{3}[(b-2+(m-1)+m']+1\\
 & = & \frac{4}{3}i+\frac{1}{3}b-m-m'
\end{eqnarray*}
by inductive hypothesis. 
Thus $\Pi _1'=\Pi _1 \cup \bigl( \cup _{k=0}^{m'} \{ \triangle r_jq_lr_{j+1} \}\bigr)$ 
is a convex decomposition of $P'$ with 
$|\Pi _1'|=u(P'')+(m'+1)\le \frac{4}{3}i+\frac{1}{3}b-m+1$. 
Now consider the covex decomposition $\Pi _0 \cup \Pi _1'$ of $P$. 
Since at least one of the pairs 
$\triangle r_0p_jq_l\in \Pi_0$ and $\triangle r_0q_lr_1 \in \Pi _1'$; 
or 
${\cal P}\in \Pi_0$ and $\triangle r_{m'}q_lr_{m'+1} \in \Pi _1'$ 
can be combined to form a single convex polygon, 
\begin{eqnarray*}
u(P) & \le & |\Pi _1'|+|\Pi _0|-1 \\
& = & \left( \frac{4}{3}i+\frac{1}{3}b-m+1 \right)+(m+1)-1 
\, \, = \, \, \frac{4}{3}i+\frac{1}{3}b+1,
\end{eqnarray*}
as desired. This is the end of Case 2. 

\bigskip
Now we may assume that 
\begin{equation}
\mbox{$|H(q)\cap B(P)| = 1$ for all $q\in B(I(P))$.}
\label{condition2}
\end{equation}
For each $j$, let $Q_j=B(P-\{p_j\})\cap I(P)$. 

\bigskip
\noindent
{\bf Case 3.}~~$|Q_j|\ge 3$ for some $j$:

Take any $j$ with $|Q_j|\ge 3$. 
For this $j$, let $m=|Q_j|$ and $q_0=p_{j-1}$. 
We label the elements of $Q_j$ as 
$q_1, \ldots , q_m$ in such a way that 
$q_0, q_1, \ldots , q_m$ occur on the boundary of ${\rm CH}(P-\{p_j\})$ 
in the counter-clockwise order, and let $q_{m+1}=p_{j+1}$ (Figure~\ref{cd05}). 
Since $m\ge 3$, $\triangle q_0q_1q_2$ and $\triangle q_{m-1}q_mq_{m+1}$ 
have disjoint interiors. 
To prove (\ref{ineq}), 
we combine two techniques which we used in the proof for Case 2. 
Let $\Pi_0= \cup _{k=0}^{m}\{ \triangle q_kp_1q_{k+1} \}$ in this case. 

\bigskip
\noindent
{\bf Case 3.1.}~~${\rm Int}(\triangle q_0q_1q_2) \cap P = 
{\rm Int}(\triangle q_{m-1}q_mq_{m+1}) \cap P = \emptyset$ 
(Figure~\ref{cd05}(a)):

Let $\Pi _1$ be a convex decomposition of $P-\{p_j, q_1, q_m \}$ with 
$u(P-\{p_j, q_1, q_m \})$ elements. 
Since each of the triangles $q_0q_1q_2$ and $q_{m-1}q_mq_{m+1}$ can be combined with 
an element of $\Pi _1$ to form a single convex polygon, 
\begin{eqnarray*}
u(P-\{p_j\}) & \le & u(P-\{p_j, q_1, q_m\}) \\
& \le & \frac{4}{3}(i-m)+\frac{1}{3}(b-1+m-2)+1 \, \, 
 = \, \, \frac{4}{3}i+\frac{1}{3}b-m. 
\end{eqnarray*}
Hence $u(P)\le u(P-\{p_j\})+|\Pi _0| = \frac{4}{3}i+\frac{1}{3}b+1$. 
%
%
\begin{center}
 \begin{figure}[htbp]
 \begin{minipage}[b]{.46\linewidth}
  \begin{center}\includegraphics[scale=.34]{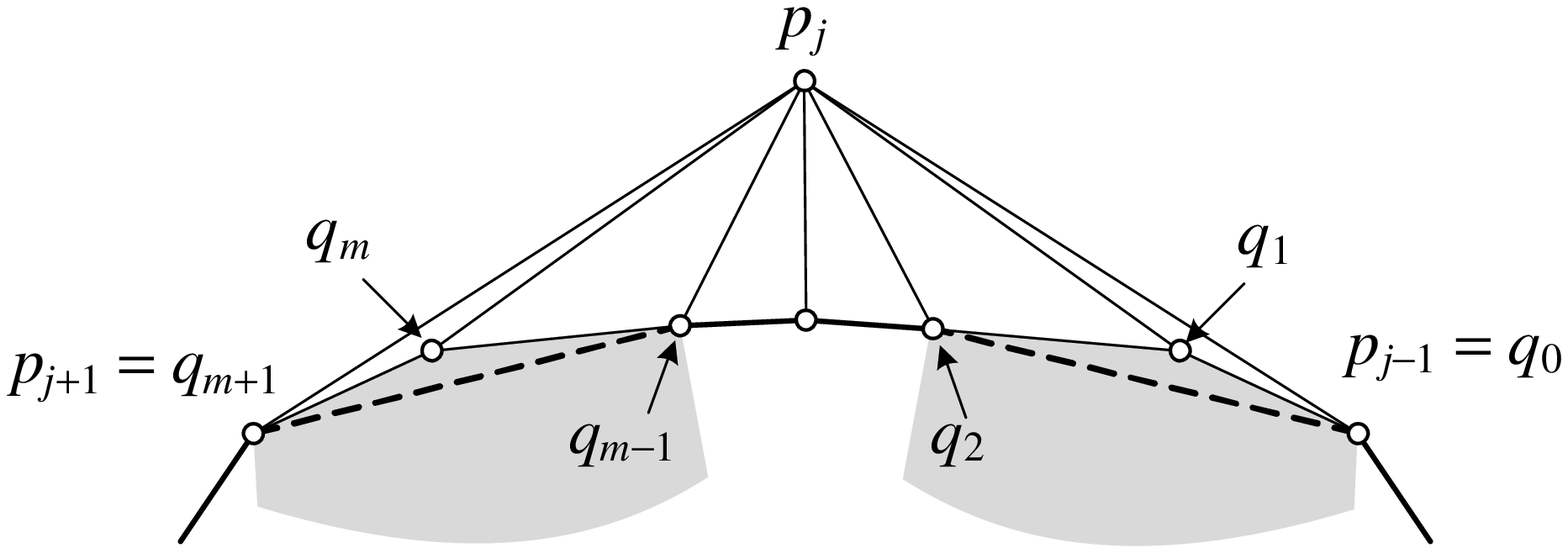}\\
  \hspace*{6.5mm}(a)
  \end{center}
  \end{minipage}
\hspace*{3mm}
  \begin{minipage}[b]{.46\linewidth}
 \begin{center}\includegraphics[scale=.34]{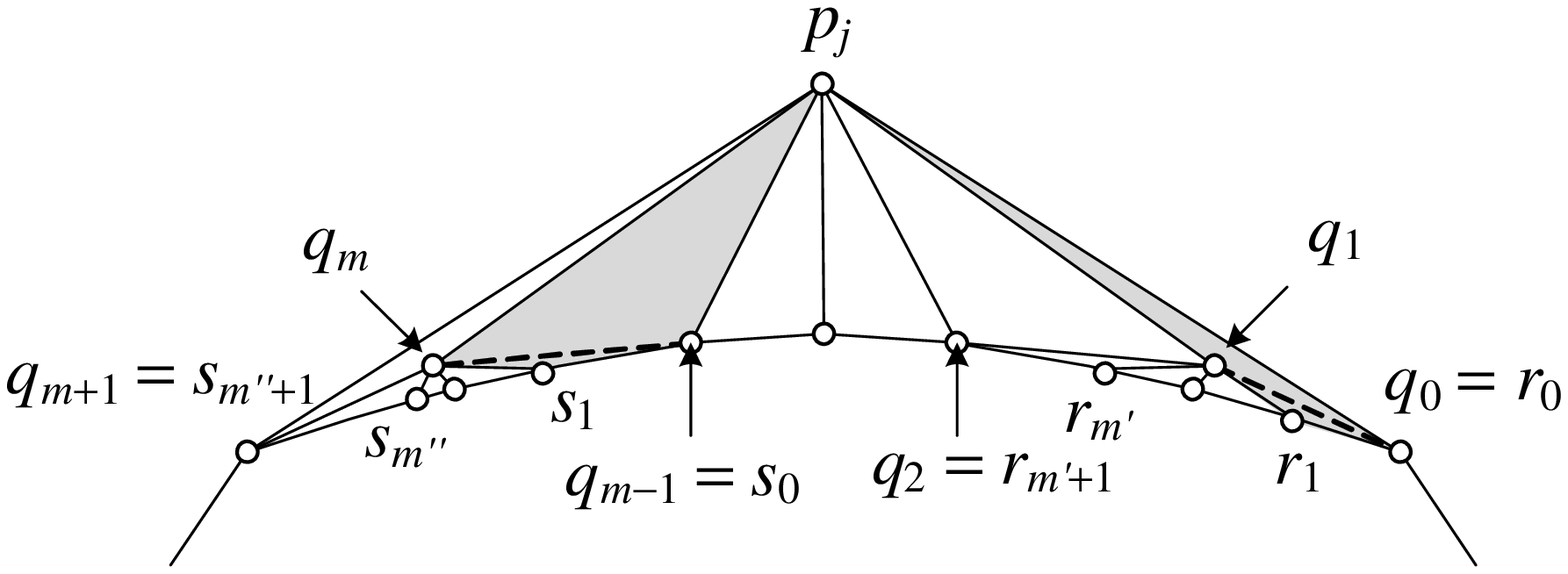}\\
  \hspace*{6.5mm}(b)
  \end{center}
  \end{minipage}
\caption{Point sets of (a)~Case 3.1 and (b)~Case 3.2.}
  \label{cd05}
  \end{figure}
\end{center}
%

\bigskip
\noindent
{\bf Case 3.2.}~~${\rm Int}(\triangle q_0q_1q_2) \cap P \ne \emptyset$ and 
${\rm Int}(\triangle q_{m-1}q_mq_{m+1}) \cap P \ne \emptyset$ 
(Figure~\ref{cd05}(b)):

Let $m'=|B(P-\{p_j,q_1\})\cap I(P-\{p_j\})|, \, r_0=q_0(=p_{j-1})$, 
label the elements of $B(P-\{p_j,q_1\})\cap I(P-\{p_j\})$ as 
$r_1, \ldots , r_{m'}$ in such a way that 
$r_0, r_1, \ldots , r_{m'}$ occur on the boundary of ${\rm CH}(P-\{p_j,q_1\})$ 
in the counter-clockwise order, and let $r_{m'+1}=q_2$ (Figure~\ref{cd05}(b)). 
Similarly, write $m''=|B(P-\{p_j,q_m\})\cap I(P-\{p_j\})|$, 
label the elements of $B(P-\{p_j,q_m\})\cap I(P-\{p_j\})$ as 
$s_1, \ldots , s_{m''}$, and let $s_0=q_{m-1}, \, s_{m''+1}=q_{m+1}$. 
Let $\Pi_1$ be a convex decomposition of $P-\{p_j,q_1,q_m\}$ 
with $u(P-\{p_j,q_1,q_m\})$ elements, 
$\Pi_2= \cup _{k=0}^{m'}\{ \triangle r_k q_1 r_{k+1} \}$, and 
$\Pi_3= \cup _{k=0}^{m''}\{ \triangle s_k q_m s_{k+1} \}$. 
Then $\Pi_1'=\Pi_1 \cup \Pi_2 \cup \Pi_3$ is a convex decompositon of 
$P-\{p_j\}$ and 
\begin{eqnarray*}
|\Pi_1'|
 & \le & \left[ \frac{4}{3}(i-m-m'-m'')+\frac{1}{3}(b-1+m-2+m'+m'')+1\right] \\
 &     & \hspace*{7cm}+ (m'+1)+(m''+1) \\
 & = & \frac{4}{3}i+\frac{1}{3}b-m+2.
\end{eqnarray*}
Now consider the convex decomposition $\Pi _1' \cup \Pi_0$ of $P$. 
Observe that at least one of the pairs $\triangle r_0p_jq_1$ and $\triangle r_0q_1r_1$; 
or $\triangle q_1p_jr_{m'+1}$ and $\triangle r_{m'}q_1r_{m'+1}$ 
can be combined to form a convex quadrilateral, 
and also 
at least one of the pairs $\triangle s_0p_jq_m$ and $\triangle s_0q_ms_1$; 
or $\triangle q_mp_js_{m''+1}$ and $\triangle s_{m''}q_ms_{m''+1}$ 
can be combined to form a convex quadrilateral. 
Thus it now follows that 
$$
u(P)\le |\Pi_1'|+|\Pi _0|-2 \le 
\frac{4}{3}i+\frac{1}{3}b+1.
$$

\bigskip
\noindent
{\bf Case 3.3.}~~Exactly one of ${\rm Int}(\triangle q_0q_1q_2) \cap P$ or 
${\rm Int}(\triangle q_{m-1}q_mq_{m+1}) \cap P$ is empty:

By symmetry, we may assume that 
${\rm Int}(\triangle q_0q_1q_2) \cap P \ne \emptyset$ 
and ${\rm Int}(\triangle q_{m-1}$
$q_mq_{m+1}) \cap P= \emptyset$ (Figure~\ref{cd05b}). 
\begin{center}
\begin{figure}[htbp]
\begin{center}
\includegraphics[scale=0.4]{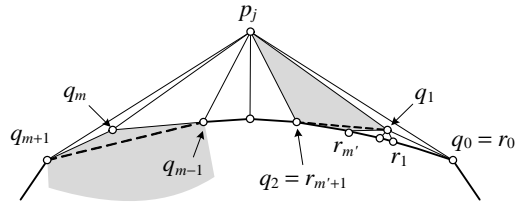}
\caption{Point sets of Case 3.3.}
\label{cd05b}
\end{center}
\end{figure}
\end{center}

In this case, we combine the arguments used in the proofs for 
Cases 3.1 and 3.2. 
As in Case 3.2, 
define $m'=|B(P-\{p_j,q_1\})\cap I(P-\{p_j\})|$, 
labels $r_1, \ldots , r_{m'} \in B(P-\{p_j,q_1\})\cap I(P-\{p_j\}), \,
r_0=q_0$ and $r_{m'+1}=q_2$. 
Now let $\Pi_1$ be a convex decomposition of $P-\{p_j,q_1,q_m\}$ with 
$u(P-\{p_j,q_1,q_m\})$ elements, and 
$\Pi_2=\cup _{k=0}^{m'}\{ \triangle r_k q_1 r_{k+1} \}$. 
By combining $\triangle q_{m-1}q_mq_{m+1}$ with an element of 
$\Pi_1$ to form a convex polygon, 
we obtain a convex decomposition $\Pi_1'$ of $P-\{p_j,q_1\}$ with 
$|\Pi _1'|=|\Pi _1|$. 
Then $\Pi_1''=\Pi _1' \cup \Pi _2$ is 
a convex decomposition of $u(P-\{p_j\})$, 
and 
\begin{eqnarray*}
|\Pi_1''| & = & |\Pi_1'|+|\Pi_2| \, \, = \, \, |\Pi_1|+|\Pi_2| \\
& \le & \left[ \frac{4}{3}(i-m-m')+\frac{1}{3}(b-1+m-2+m')+1 \right]
+(m'+1)\\
& = & \frac{4}{3}i+\frac{1}{3}b-m+1.
\end{eqnarray*}
Since at least one of the pairs $\triangle r_0p_jq_1$ and 
$\triangle r_0q_1r_1$; 
or $\triangle q_1p_jr_{m'+1}$ and $\triangle r_{m'}q_1r_{m'+1}$ 
can be combined to form a convex quadrilateral, 
$$
u(P)\le |\Pi_1''|+|\Pi _0|-1 \le 
\frac{4}{3}i+\frac{1}{3}b+1.
$$

\bigskip
\noindent
{\bf Case 4.}~~$|Q_j|\le 2$ for all $j$: 

In the rest of this paper, 
we denote by $q_{j}$ (resp.\,$q_{j}'$) the element of $Q_j$ 
which occurs immediately after $p_{j-1}$ (resp.\,before $p_{j+1}$) 
on the boundary of ${\rm CH}(P-\{p_j\})$, 
in the counter-clockwise order 
(Figure~\ref{cd09}; we have $q_j=q_j'$ when $|Q_j|=1$). 
\begin{center}
\begin{figure}[htbp]
\begin{center}
\includegraphics[scale=0.27]{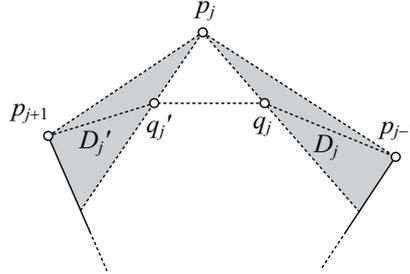}
\caption{$q_{j}, \, q_{j}'$ and $D_j, \, D_j'$.}
\label{cd09}
\end{center}
\end{figure}
\end{center}
For three points $a, b$ and $c$ 
in general position on the plane, let $A(a; b, c)$ denote the 
angular domain bounded by half-lines $ab$ and $ac$, and containing 
the segment $bc$. 
We let $D_j={\rm CH}(P) \cap {\rm Int}(A(p_j; q_j, p_{j-1}))$ 
and $D_j'={\rm CH}(P) \cap {\rm Int}(A(p_j; p_{j+1}, q_j'))$. 
\begin{lemma}
\mbox{$D_j \cap P = D_j' \cap P=\emptyset $ 
for all $j$.}
\label{dj}
\end{lemma}
\medskip
\noindent
{\em Proof.}~We prove only $D_j \cap P=\emptyset $ 
since $D_j' \cap P=\emptyset $ can be proved quite similarly. 
Suppose there exists $p\in D_j \cap P$. 
Then $p\in D_j\setminus \triangle p_{j-1}p_jq_j$ by the definition of $q_j$. 
Among such $p$'s, take the one such that $\angle pq_jp_{j-1}$ is the smallest. 
Then if $p\in B(P)$ (i.e. $p=p_{j-2}$), then we have 
${\rm Int}(\triangle p_{j-2}p_{j-1}p_{j}) \cap P = \emptyset $, which 
contradics (\ref{1}); 
and if $p\in I(P)$, then we have $q_{j}=p^+$ and 
$H(p)\cap B(P) \supset \{p_{j-1}, p_{j}\}$, 
which contradics (\ref{condition2}).
\qed

\medskip
Since $i\ge 3$ (recall (\ref{0})), Lemma~\ref{dj} in particular 
implies that $q_j\ne q_j'$ for all $j$, 
and hence $|Q_j|=2$ for all $j$. 
It also follows from Lemma~\ref{dj} that $q_{j}'=q_{j+1}$ for all $j$. 

\medskip
\noindent
{\bf Case 4.1.}~~
${\rm Int}(\triangle q_{j}q_{j+1}q_{j+2}) \cap P=\emptyset$ 
for some $j$: 

For this $j$, assume further that 
\begin{equation}
{\rm Int}(\triangle p_{j-1}q_{j}q_{j+2}) \cap P \ne \emptyset 
\quad
\mbox{(Figure~\ref{cd10}(a); we have $b\ge 4$ in this case)}.
\label{further1}
\end{equation}
\begin{center}
 \begin{figure}[htbp]
 \begin{minipage}[b]{.46\linewidth}
 \, \, \, 
  \begin{center}\includegraphics[scale=.27]{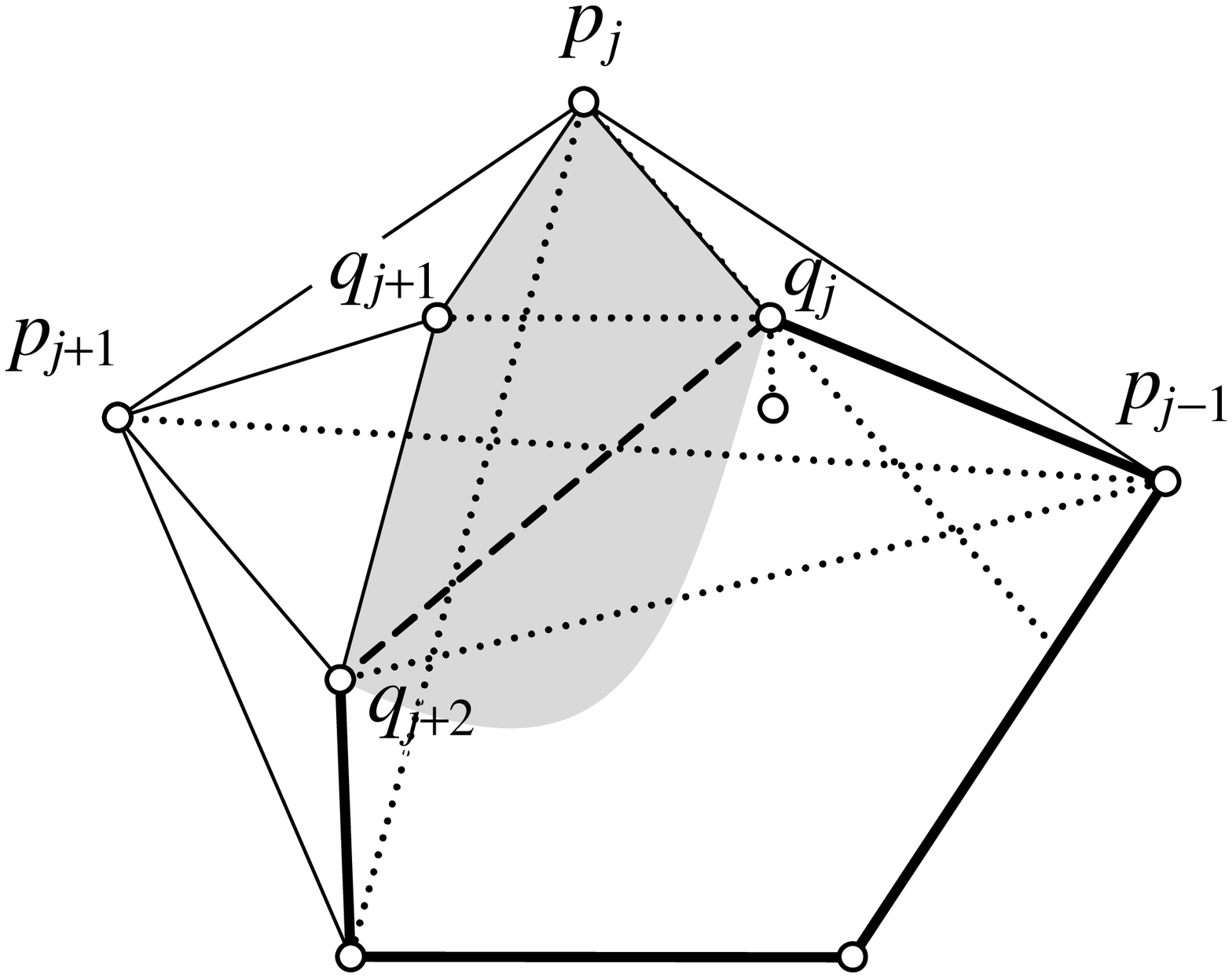}\\
  (a)~${\rm Int}(\triangle p_{j-1}q_{j}q_{j+2}) \cap P \ne \emptyset$
  \end{center}
  \end{minipage}
   \begin{minipage}[b]{.46\linewidth}
 \, \, \, \, \, \, \, \, \, 
  \begin{center}\includegraphics[scale=.27]{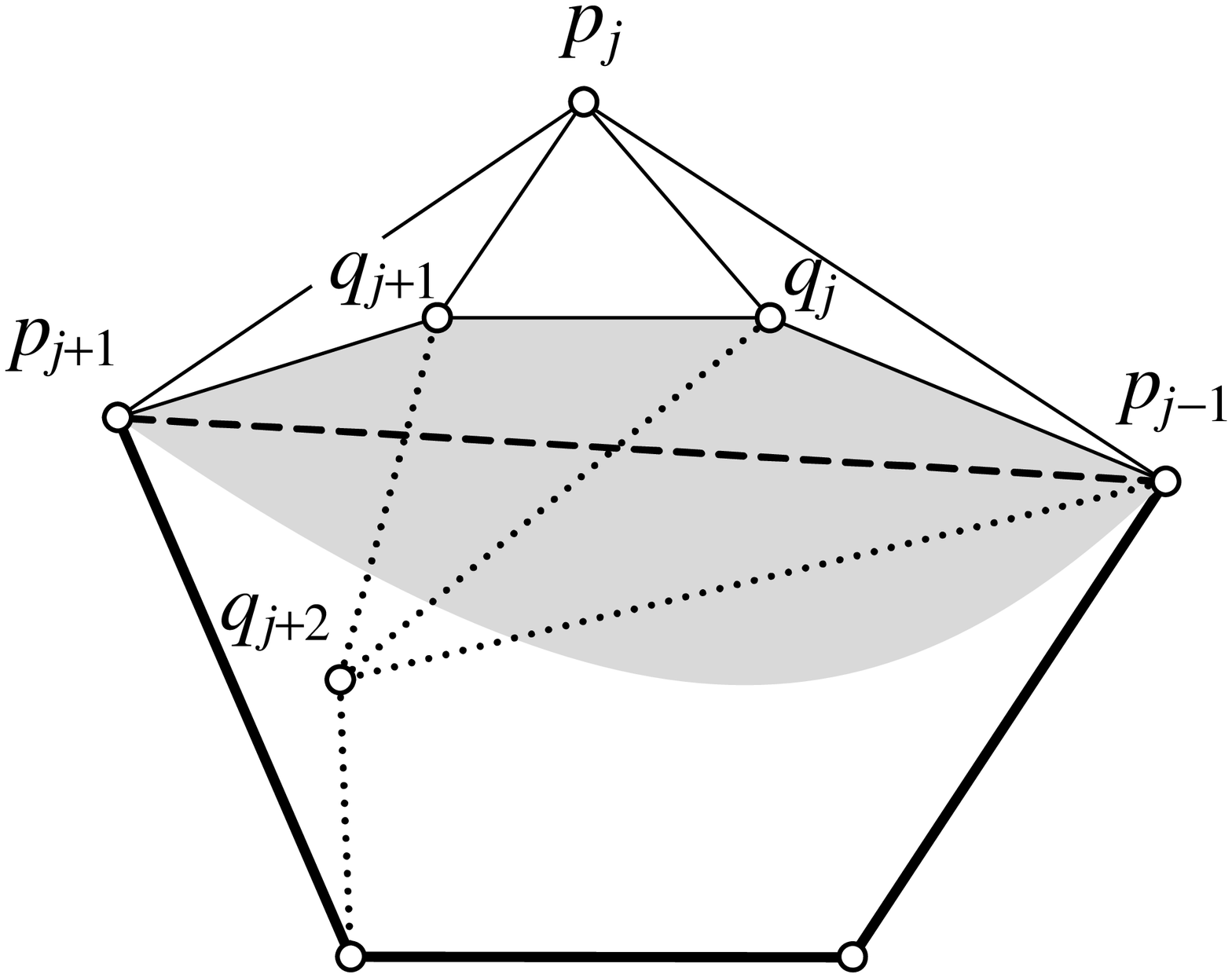}\\
  (b)~${\rm Int}(\triangle p_{j-1}q_{j}q_{j+2}) \cap P = \emptyset$
  \end{center}
  \end{minipage}
  \caption{Two subcases of Case 4.1.}
  \label{cd10}
  \end{figure}
\end{center}
%
Take a convex decomposition of $P-\{p_j, p_{j+1}, q_{j+1}\}$ 
with $u(P-\{p_j, p_{j+1}, q_{j+1}\})$ elements, and let 
${\cal P}$ be its element that conntains the segment $q_jq_{j+2}$ 
on its boundary. 
\begin{lemma}
All vertices, except $q_j$, of ${\cal P}$ lie on 
the same side as $q_{j+2}$ with respect to the straight line $p_jq_j$. 
\label{cl1}
\end{lemma}
\begin{proof}
By Lemma~\ref{dj}, $p_{j-1}$ is the only element of $P$ 
which lies on the opposite side of $q_{j+2}$ 
with respect to the line $p_jq_j$. 
But we have $p_{j-1}\not\in {\cal P}$ 
since otherwise ${\cal P}\supseteq \triangle p_{j-1}q_{j}q_{j+2}$, 
and hence ${\rm Int}({\cal P})\cap (P-\{p_j, p_{j+1}, q_{j+1}\}) = 
{\rm Int}({\cal P})\cap P \ne \emptyset$ by 
(\ref{further1}), which contradicts the choice of ${\cal P}$. 
\end{proof}

It now follows from Lemma~\ref{cl1} that ${\cal P}$ can be combined with 
the quadrilateral $p_jq_{j+1}q_{j+2}q_{j}$ to form a single convex polygon. 
Thus 
$$
u(P)  \le  u(P-\{p_j, p_{j+1}, q_{j+1}\}) + 4
\le  \left[ \frac{4}{3}(i-3)+\frac{1}{3}b+1 \right] +4 
=  \frac{4}{3}i+\frac{1}{3}b+1. 
$$

Next assume that 
\begin{equation}
{\rm Int}(\triangle p_{j-1}q_{j}q_{j+2}) \cap P = \emptyset
\quad \mbox{(Figure~\ref{cd10}(b))}.
\label{further2}
\end{equation}

First consider the case where $b\ge 4$. 
In this case, we can verify that quadrilateral $p_{j-1}q_{j}q_{j+1}p_{j+1}$ 
contains no element of $P$ in its interior. 
Take a convex decomposition of $P-\{p_j, q_{j}, q_{j+1}\}$ 
with $u(P-\{p_j, q_{j}, q_{j+1}\})$ elements, and let 
${\cal P}$ be its element that conntains the segment $p_{j-1}p_{j+1}$ 
on its boundary. 
Since the convex quadrilateral $p_{j-1}q_{j}q_{j+1}p_{j+1}$ can be 
combined with ${\cal P}$ to form a single convex polygon, 
$$
u(P)  \le  u(P-\{p_j, q_{j}, q_{j+1}\}) + 3
      \le  \left[ \frac{4}{3}(i-2)+\frac{1}{3}(b-1)+1 \right] +3
      =  \frac{4}{3}i+\frac{1}{3}b+1. 
$$

Next consider the case where $b=3$. 
We have $P=\{p_0, p_1, p_2, q_0, q_1, q_2\}$ in this case, 
and we can easily verify that 
$u(P)=6$, which is equal to 
$\frac{4}{3} \! \cdot \! 3+\frac{1}{3} \! \cdot \! 3+1=
\frac{4}{3}i+\frac{1}{3}b+1$, as desired. 

\bigskip
\noindent
{\bf Case 4.2.}~~
${\rm Int}(\triangle q_{j}q_{j+1}q_{j+2}) \cap P \ne \emptyset$ 
for all $j$ (Figure~\ref{cd10c}): 

Take a convex decomposition of $I(P)$ 
with $u(I(P))$ elements, and 
for each $j=0, 1, \dots , b-1$, let ${\cal P}_j$ be its element 
that contains the segment $q_{j}q_{j+1}$ 
on its boundary. 
%
\begin{center}
\begin{figure}[htbp]
\begin{center}
\includegraphics[scale=0.27]{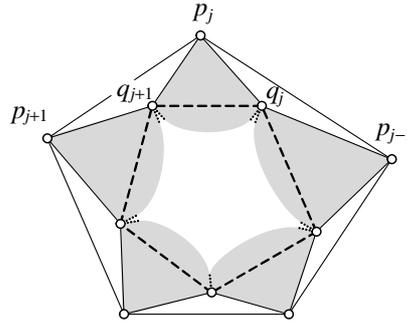}
\caption{A point set of Case 4.2.}
\label{cd10c}
\end{center}
\end{figure}
\end{center}

By the assumption of this case, ${\cal P}_j \ne {\cal P}_{j+1}$ for all $j$. 
Furthermore, it follows from Lemma~\ref{dj} that 
$\triangle p_jq_{j+1}q_j$ and ${\cal P}_j$ can be combined 
to form a single convex polygon. 
It is possible that 
${\cal P}_{j_1}={\cal P}_{j_2}= \cdots ={\cal P}_{j_k}$ for 
some non-consecutive indices $j_1, \, \dots , \, j_k$. 
In this case, we can combine 
triangles $p_{j_1}q_{j_1+1}q_{j_1}, \, \dots , \, p_{j_k}q_{j_k+1}q_{j_k}$ 
and ${\cal P}_{j_1}(=\dots ={\cal P}_{j_k})$ 
to form a single convex polygon. 
As a consequent of it, we obtain
$$
u(P)  \le  u(I(P)) + b 
      \le  \left[ \frac{4}{3}(i-b)+\frac{1}{3}b+1 \right] +b 
      <  \frac{4}{3}i+\frac{1}{3}b+1, 
$$
which completes the proof of Theorem 1. 
\qed


\begin{thebibliography}{5}
\itemsep=2pt
%
\bibitem{4}
O. Aichholzer and H. Krasser, 
The point set order type data base: A collection of applications and results, 
{\sl Proc. 13th Canadian Conference on Computational Geometry}, 
University of Waterloo, Waterloo, 2001, 17--20.
%
\bibitem{5}
J. Garc\'{\i}a-L\'opez, C.M. Nicol\'as, 
Planar point sets with large minimum convex partitions, 
{Proc. 22nd European Workshop on Computational Geometry}, 
Delphi, 2006, 51--54.
%
\bibitem{3}
K. Hosono, 
On convex decompositions of a planar point set, 
{\sl Discrete Mathematics} {\bf 309}(6) (2009), 1714--1717.
%
\bibitem{2}
V. Neumann-Lara, E. Rivera-Campo, J. Urrutia, A note on convex 
decompositions 
of point sets in the plane, {\sl Graphs and Combinatorics} {\bf 20}(2) (2004), 
223--231�D
%
\bibitem{1}
J. Urrutia, Open problem session, {\sl 10th Canadian Conference 
on Computational Geometry}, McGill University, Montreal, 1998.
\end{thebibliography}
\end{document}